\documentclass{amsart}
\usepackage{amsmath}
\usepackage{mathtools}

\theoremstyle{theorem}
\newtheorem{theorem}{Theorem}
\newtheorem{proposition}[theorem]{Proposition}

\usepackage{tikz}
\usepackage{pgfplots}
\usepackage{hyperref}

\usepackage[bibtex-style]{amsrefs}

\newcommand*\circled[1]{\tikz[baseline=(char.base)]{
            \node[shape=circle,draw,inner sep=1pt] (char) {#1};}}

\renewcommand{\th}{{\mbox{\scriptsize{th}}}}

\begin{document}

\title{Collatz meets Fibonacci}
\markright{Collatz meets Fibonacci}
\author{Michael Albert, Bjarki Gudmundsson and Henning Ulfarsson}

\begin{abstract}
The Collatz map is defined for a positive even integer as half that integer,
and for a positive odd integer as that integer threefold, plus one. The Collatz conjecture
states that when the map is iterated the number one is eventually reached.

We study permutations that arise as sequences from this iteration. We show that
permutations of this type of length up to $14$ are enumerated by the Fibonacci
numbers. Beyond that excess permutations appear. We will explain the
appearance of these excess permutations and give an upper bound on the exact
enumeration.
\end{abstract}

\maketitle

\section{Introduction} 
\label{sec:introduction}

A well-known conjecture due to Lothar Collatz from 1937 states that when the function
\[
f(x) = \left \{
  \begin{array}{ll}
    x/2 & \mathrm{if\ } x \equiv 0 \pmod{2} \\
    3x+1 & \mathrm{if\ } x \equiv 1 \pmod{2}
  \end{array}
\right.
\]
is iterated from an initial positive integer $x$ we eventually reach
the cycle $(1,4,2)$. Although many have tried to prove the conjecture, or find a counterexample
to it, it remains open to this day. In this paper we assume that the conjecture is true.
A collection edited by Lagarias \cite{MR2663745} contains many articles related to the conjecture and its generalizations, and an annotated bibliography. One of the intriguing features of this function is that, starting from any initial value $x$, the sequence of iterates: $x, f(x), f(f(x)), \dots$, behaves at first irregularly before its eventual apparent inevitable decline from some power of 2 down to the three element cycle. In this paper we consider the initial interesting phase of these iterations when seen from a distance, i.e., only with regard to their relative values, and not the actual numbers produced.

For example, starting the iteration at 12 provides the sequence:
\[
\underbrace{12,\, 6,\, 3,\, 10,\, 5}_{\text{trace}},\, 16,\, 8,\, 4,\, 2,\, 1,\, 4,\, 2,\, 1,\, \dotsc.
\]
The sequence of numbers up until the first power of two is the interesting phase of the iteration, which we will call its  \emph{trace}. The elements of a trace are all distinct,
and viewed from a distance we might replace each element of a trace by its rank (i.e., the $i^\th$ smallest number of the trace is replaced by $i$). The resulting sequence is a permutation, and we shall call permutations produced in this manner \emph{Collatz permutations}. We denote the Collatz permutation obtained from initial value $x$ by $C(x)$. So we have: $C(12) = 5\,3\,1\,4\,2$. Clearly, the map $C$ is not one to one. For instance:
\[
1 = C(5) = C(21) = C(85) = \cdots = C((2^{2k}-1)/3) = \cdots
\]
However, as the length of $C(x)$ increases, coincidences become more rare, e.g., the only other $x \leq 1000$ with $C(x) = C(12)$ is $x = 908$. It seems natural to ask:
\emph{Among the permutations of length $n$, how many are Collatz permutations?}

Considering only those $x \leq 10^8$ for which the length of $C(x)$ is at most 7 produces Table~\ref{tab:1-7}.
\begin{table}[ht]
\centering
\begin{tabular}{cc}
Length & Collatz permutations \\
\hline
1 & 1 \\
2 & 1 \\
3 & 2 \\
4 & 3 \\
5 & 5 \\
6 & 8 \\
7 & 13
\end{tabular}
\caption{Number of Collatz permutations of length less than 8 (experimentally).}
\label{tab:1-7}
\end{table}
As the reader will have noticed the values in the table are the Fibonacci numbers. In what follows we will explain this phenomenon and show that it persists through length $14$. Beyond that point, excess permutations appear -- and we will explain how and why this occurs as well.

\section{Types of traces} 
\label{sec:types_of_traces}
The appearance of the Fibonacci numbers in the enumeration of Collatz permutations is easy to explain. The steps that occur in a trace are of two types depending on the parity of the argument. We can call them: \emph{up steps} ($x \mapsto 3x+1$ when $x$ is odd) denoted by $u$; and \emph{down steps} ($x \mapsto x/2$ when $x$ is even) denoted by $d$. Two up steps can never occur consecutively since $3x+1$ is even when $x$ is odd\footnote{Of course this is why some authors consider the ``fundamental'' steps of the Collatz iteration to be $x \mapsto (3x+1)/2$ for odd $x$, and $x \mapsto x/2$ for even $x$.}. The step types in a trace can be recovered from the resulting Collatz permutation according to the pattern of rises and descents. We call the resulting sequence of $u$'s and $d$'s the \emph{type} of the trace. As well as not containing consecutive $u$'s, the last symbol in such a sequence must be a $d$ (since there is a ``hidden'' $u$ occurring next to take us to a power of 2). As is well known, the number of such sequences of length $n$ is given by $F_n$, the $n^\th$ Fibonacci number (with $F_1 = F_2 = 1$ and $F_n = F_{n-1} + F_{n-2}$ for $n > 1$).

So, in order to show that there are at least $F_n$ Collatz permutations of length $n$ it will be enough to show that any sequence of $u$'s and $d$'s satisfying the necessary conditions above actually occurs as the type of some trace. To prove this, it is helpful to turn our focus to the end of a trace, rather than its beginning.


The final element of any trace is always a number of the form $(A-1)/3$ where $A = 2^a$ (of course $a$ must be even but we will be making further requirements shortly, so ignore this for the moment). Think of $u$ and $d$ as functions and let their inverses be denoted $U$ and $D$, that is:
\[
U(x) = u^{-1}(x) = (x-1)/3 \quad \text{ and } \quad D(x) = d^{-1}(x) = 2x.
\]
A \emph{witness} for a type, $\sigma$, is an $A = 2^a$ such that there is a trace ending at $(A-1)/3$ with type $\sigma$. For example, to find a witness for the type $uddud$ requires at least that $UDDUDU(A)$ be an integer. Note here that the ``hidden $u$'' has become explicit, and that the order of $U$'s and $D$'s in the functional application is the same as that in the type -- because the type is written in the reverse of ``normal'' function application order! Now
\[
UDDUDU(A) = \frac{8 A - 29}{27}.
\]
In order for this to be an integer requires $A = 7 \pmod{27}$, and remembering that $A = 2^a$, the least solution to this is $A = 2^{16} = 65536$. Unraveling the applications of $U$'s and $D$'s leads to $x = 19417$. The trace of $x$ is
\[
19417,\, 58252,\, 29126,\, 14563,\, 43690,\, 21845,
\]
and $C(x) = 2\,6\,4\,1\,5\,3$.

We now check that, given a type $\sigma$, it is enough to find a witness $A$ in this way (i.e., that the initial necessary integrality condition is also sufficient) and that infinitely many witnesses always exist.

\begin{proposition}
\label{pr:atLeastOne}
If a type $\sigma$ contains $k$ $u$'s then there is a single congruence of the form $A = c \pmod{3^{k+1}}$ which must be satisfied in order that a trace of type $\sigma$ ends with witness $A$. Consequently, there is a least witness $A = 2^a$ with $a \leq 2 \cdot 3^{k}$, and a general witness is of the form $2^{a + jd}$ where $j$ is a nonnegative integer and $d = 2 \cdot 3^{k}$.
\end{proposition}


\begin{proof}
From $\sigma$ construct the sequence, $\Sigma$ of $U$'s and $D$'s obtained by changing each $u$ to $U$ and each $d$ to $D$ and then appending a $U$. Now compute $\Sigma(A)$ formally obtaining an expression of the form:
\[
\frac{2^m A - b}{3^{k+1}}
\]
where $m$ is the number of $d$'s in $\sigma$ and $k$ the number of $u$'s. In order that $A$ be a witness to a trace of type $\sigma$ it is necessary that this be an integer which, since $2$ and $3^{k+1}$ are relatively prime, leads to a necessary condition of the form $A = c \mod{3^{k+1}}$. To see that this condition is also sufficient suppose inductively, that $\Gamma(A) \in \mathbb{Z}$ for some suffix $\Gamma$ of $\Sigma$. If $\Gamma$ begins with $D$ then formally $\Gamma (A) = (2^k A - 2 b)/3^m$ for some $k > 0$ and integer $b$. If this is an integer then so is $d \Gamma(A) = (2^{k-1} A - b)/3^m$ which is the next element of the trace we are trying to construct. The case when $\Gamma$ begins with $U$ is easier since $u \Gamma(A) = 3 \Gamma(A) + 1$ and so is still an integer as well. To finish, note that it is well known that 2 is a primitive root modulo $3^k$ for all $k$, and the remaining claims are simple consequences of that fact.
\end{proof}

%
%

This proposition shows that every potential type has a witness and thereby proves that there are at least $F_n$ Collatz permutations of length $n$.

Table~\ref{tab:1-21} shows that there are exactly $F_n$ Collatz permutations of length $n$ for $n = 1,2, \dotsc, 14$ but for greater $n$ there are more.
		\begin{table}[ht]
			\centering
			\begin{tabular}{cc|ccc}
		    length & \#perms & length & \#perms & excess \\
		    \hline
		    $1$  & $1$   & $15$ & $611$   & $1$ \\
		    $2$  & $1$   & $16$ & $989$   & $2$ \\
		    $3$  & $2$   & $17$ & $1600$  & $3$ \\
		    $4$  & $3$   & $18$ & $2587$  & $3$ \\
		    $5$  & $5$   & $19$ & $4185$  & $4$ \\
		    $6$  & $8$   & $20$ & $6771$  & $6$ \\
		    $7$  & $13$  & $21$ & $10953$ & $7$ \\
		    $8$  & $21$  & $22$ & $17720$ & $9$ \\
			$9$  & $34$  & $23$ & $28669$ & $12$ \\
			$10$ & $55$  & $24$ & $46383$ & $15$ \\
			$11$ & $89$  & $25$ & $75044$ & $19$ \\
			$12$ & $144$ & $26$ & $121417$ & $24$ \\
			$13$ & $233$ & $27$ & $196448$ & $30$ \\
			$14$ & $377$ & $28$ & $317850$ & $39$ \\
			     &       & $29$ & $514278$ & $49$ \\
			     &       & $30$ & $832101$ & $61$ \\
			     &       & $31$ & $1346346$ & $77$ \\
			     &       & $32$ & $2178405$ & $96$
			\end{tabular}
			\caption{Number of Collatz permutations of length less than $32$}
			\label{tab:1-21}
		\end{table}

	The shortest type that is associated with more than one Collatz permutation
    is the type $\sigma = uddudududduddd$ which has the integrality condition
	$2^a = 16 \bmod 729$. The smallest solution to this equation is $a = 4$
    corresponding to the trace $9, 28, 14, 7, 22, 11, 34, 17, 52, 26, 13, 40, 20, 10, 5$ and the first permutation in Figure~\ref{tab:firstExcessPerms}.
	However, the next solution to the integrality condition is $a = 490$. The initial number corresponding to this solution produces a different permutation than
	the smaller initial number, and is shown in the second line in Figure~\ref{tab:firstExcessPerms}.
		\begin{figure}[ht]
		    \centering
		    \begin{tabular}{ccccccccccccccc}
		        \circled{$3$} & $12$ & $7$ & $2$ & $10$ & $5$ & $13$ & $8$ & $15$ & $11$ & $6$ & $14$ & $9$ & \circled{$4$} & $1$ \\
		        \circled{$4$} & $12$ & $7$ & $2$ & $10$ & $5$ & $13$ & $8$ & $15$ & $11$ & $6$ & $14$ & $9$ & \circled{$3$} & $1$
		    \end{tabular}
		    \caption{Two different permutations associated with the type $uddudududduddd$}
			\label{tab:firstExcessPerms}
		\end{figure}

	In the next section we will explain why this type gives us
	two different permutations and when this should be expected.


\section{Excess permutations} 
\label{sec:excess_permutations}
How can one type correspond to different permutations? Consider the type $\sigma = dududd$. If we are considering a potential witness $t$ for it, then the trace that would be generated is:
\[
DUDUDDU(t),\,UDUDDU(t),\,DUDDU(t), \ldots ,DU(t),\,U(t).
\]
Thought of as linear functions of $t$ these lines are shown in Figure~\ref{fig:lines}. We find a witness for the type wherever we find a vertical line $t = A$ where $A = 2^a$ and all the intersection points of $t=A$ with these lines are at integer heights. We can then read off the relative order of the lines at this point according to the order they are crossed as we move up the line $t = A$ from the $t$-axis.

\begin{figure}
	\centering
	\begin{tikzpicture}[scale=0.5]
	    \draw [<->] (0,10) -- (0,0) -- (17,0) node[right] {$t$};
	    \draw [->] (0,0) -- (0,-2);

	    \draw[domain=0:17] plot (\x, {1/3 * \x - 1/3}) node[right,above] {$\scriptstyle 7$};
	    \draw[domain=0:16] plot (\x, {2/3 * \x - 2/3}) node[right,below] {$\scriptstyle 6$};
	    \draw[domain=0:7.8333] plot (\x, {4/3 * \x - 4/9}) node[right,below] {$\scriptstyle 5$};
	    \draw[domain=0:17] plot (\x, {4/9 * \x - 7/9}) node[right,above] {$\scriptstyle 4$};
	    \draw[domain=0:13] plot (\x, {8/9 * \x - 14/9}) node[right,below] {$\scriptstyle 3$};
	    \draw[domain=0:17] plot (\x, {8/27 * \x - 23/27}) node[right,above] {$\scriptstyle 2$};
	    \draw[domain=0:17] plot (\x, {16/27 * \x - 46/27}) node[right,above] {$\scriptstyle 1$};

	    \draw[red,dashed] (15,10) -- (15,0) node[below] {$2^8$};

	    \draw[blue,thick] (25/4,10) -- (25/4,0) node[below] {$6.25$};

	\end{tikzpicture}
	\caption{Linear functions determining permutations associated with the type
	$dududd$. The labels on each line correspond to the point that the corresponding element would occur in a trace and so the permutation associated with the intersections on the line $x = 2^8$ is $4\,1\,6\,3\,7\,5\,2$.}
	\label{fig:lines}
\end{figure}
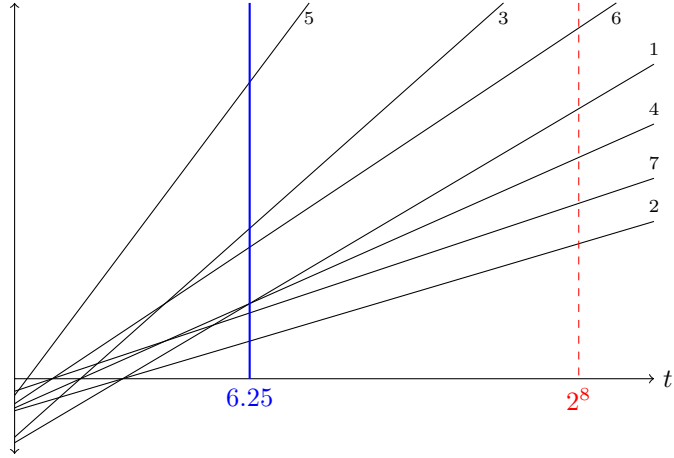

Consequently we see that if we have two potential witnesses such that there is no intersection point between two lines in the corresponding family lying between them, that they will determine the same permutation. On the other hand, if there were a witness between every pair of intersections of the lines for a type, and if these intersections all occurred at distinct abscissae, we might have up to $n \choose 2$ witnesses for any given type producing distinct permutations. However, we can rule out such a wealth of witnesses quite easily as we can show that the second potential witness always lies to the right of the rightmost intersection point among the lines. That is:

\begin{proposition}
\label{pr:atMostTwo}
For any type $\sigma$ there are at most two distinct permutations $C(x)$ arising from $x$ of type $\sigma$.
\end{proposition}

\begin{proof}
Suppose that $\sigma$ contains $k$ occurrences of $u$. The values of $a$ for which $A = 2^a$ are witnesses form an arithmetic progression with common difference $2 \cdot 3^k$, and in particular, the second witness is greater than $2^{2 \cdot 3^k}$. On the other hand we can crudely bound the greatest abscissa of a point of intersection between the lines arising from $\sigma$ as follows. Suppose that we have two such lines, one resulting from a suffix of $\Sigma$ (notation as above) containing $p$ occurrences of $D$ and $q$ occurrences of $U$ (call this the $pq$-line), and the other from a longer suffix containing $r$ occurrences of $D$ and $s$ occurrences of $U$ (call this the $rs$-line).  Suppose that the $y$ intercept of the $pq$-line is less than that of the $rs$-line. The $y$-intercept of the $pq$-line is certainly greater than $-2^{p-q}$, while that of the $rs$-line is less than or equal to $0$. So their intersection occurs to the left of the $x$-value determined by:
\[
\frac{2^p}{3^q} x - 2^{p-q} = \frac{2^r}{3^s} x
\]
which is
\[
x = \frac{2^{p-q}}{2^p - 2^r 3^{s-q}} 3^s
\]
In the other case, the $rs$-line has lesser $y$-intercept than the $pq$-line. Using similar bounds gives an intersection point to the left of
\[
x = \frac{2^{r-s}}{2^r - 2^p 3^{s-q}} 3^s.
\]
In both estimates the denominator is a non-zero integer multiple of $2^p$ (since $p \leq r$ and $q \leq s$ and at least one of these inequalities is strict) and hence the first factors are bounded by $2^{-q} 3^s$ and $2^{r-s -p} 3^s$ respectively. Therefore, $x \leq 3^s \leq 3^k < 2^{2 \cdot {3^k}}$.
\end{proof}

Recall that the type $\sigma = uddudududduddd$ had witnesses
$2^4, 2^{490}, \dotsc$. The greatest abscissa of an intersection point for the lines arising from this sequence is approximately
$44.04$. The first two witnesses are clearly on different sides of this intersection, explaining why we get two different Collatz permutations. 

Propositions \ref{pr:atLeastOne}  and \ref{pr:atMostTwo} show that the number of types of length $n$ lies between $F_n$ and $2 F_n$. Moreover, the arguments leading up to Proposition \ref{pr:atMostTwo} allow one to compute, given a type whether it corresponds to one or two types, and it is by these means that the data of Table \ref{tab:1-21} were derived.

To get an exact enumeration of the Collatz permutations one would need to understand which types are associated with two permutations. We call these \emph{excess creating types} (\emph{ET's}). Given an ET we can always create a new ET by prepending a $d$. This is because the extra $d$ does not alter the integrality condition and can only increase the maximum intersection point. This at least shows that the number of ET's is non-decreasing.

\section{Future work} 

As with so many aspects of the whole Collatz disease, a few answers just seem to lead to more questions. 

\begin{itemize}
\item
\item
How exactly does the number $e_n$ of ET's of length $n$ behave? The data above suggests that it might be something like ``half Fibonacci rate'' i.e., roughly proportional to the square root of the $n^\th$ Fibonacci number, or equivalently satisfy ing $e_n \sim e_{n-2} + e_{n-4}$.
\item
In our current dataset we always get $c = 16$ in the integrality conditions for ET's. This is probably just a result of our search procedure having explored only ``small'' values -- but perhaps it might be a necessary condition.
\item
We have an extrinsic way of creating the Collatz permutations: run the Collatz process and see what comes out. Is there an intrinsic way to recognize these permutations, beyond the obvious condition that they cannot contain consecutive rises?
\item
There are several other maps similar to the Collatz map and there is also the modified Collatz sequence (where $u$ is replaced with $ud$), as well as the Syracuse function where long down-steps are collapsed into one down-step. How do these analyses transfer to those contexts?
\end{itemize}

For the reader who is eager to start exploring Collatz permutations, we have a small code library supporting some basic features for working with their types. The code is written in the Sage open-source mathematics software system, but should run in Python with minor modifications. The code can be found at \url{https://github.com/SuprDewd/CollatzPermutations}.

\begin{section}{Acknowledgment.}
The authors are grateful to William Stein for access to the Sage Combinat
Cluster, supported by NSF Grant No.~DMS-0821725.
\end{section}

\end{document}